\documentclass[11pt,reqno,fleqn]{article}
\usepackage{amsfonts,amssymb,amsmath,amsthm}
\usepackage[top=20mm,bottom=20mm,left=20mm,right=20mm]{geometry}
\usepackage[T1]{fontenc}
\usepackage{times}
\usepackage{multicol}
\usepackage{amsfonts}
\usepackage{graphicx}
\usepackage{graphics}
\usepackage{enumitem}
\usepackage{relsize}
\usepackage{cite}
\usepackage{fancyhdr}
\usepackage{color}
\fancypagestyle{titlepage}
{
   \fancyhf{}
   \lhead{The final publication is available at IOS Press through http://dx.doi.org/10.3233/JIFS-161429}
   \rhead{}
   \fancyfoot[C]{\thepage}
}

\newcommand\blfootnote[1]{%
  \begingroup
  \renewcommand\thefootnote{}\footnote{#1}%
  \addtocounter{footnote}{-1}%
  \endgroup
}
\newtheorem{thm}{Theorem}[section]
\numberwithin{equation}{section}
\newtheorem{defin}[thm]{Definition}

\newtheorem{cor}[thm]{Corollary} 
\newtheorem{lem}[thm]{Lemma} 
\newtheorem{pr}[thm]{Proposition}
\newtheorem{rem}[thm]{Remark}
\theoremstyle{definition}
\newtheorem{ex}[thm]{Example}
\providecommand{\ceil}[1]{ \lceil #1 \rceil }
\usepackage{etoolbox}
\makeatletter
\patchcmd{\@maketitle}{\begin{center}}{\begin{flushleft}}{}{}
\patchcmd{\@maketitle}{\begin{tabular}[t]{c}}{\begin{tabular}[t]{@{}l}}{}{}
\patchcmd{\@maketitle}{\end{center}}{\end{flushleft}}{}{}
\makeatother
\begin{document}

\title{Euler summability method of sequences of fuzzy numbers and a Tauberian theorem}
\author{Enes Yavuz}
\date{{\small Department of Mathematics, Manisa Celal Bayar University, Manisa, Turkey.\\ E-mail: enes.yavuz@cbu.edu.tr}}
\maketitle
\thispagestyle{titlepage}
\blfootnote{\emph{Key words and phrases:} Sequences of fuzzy numbers, Euler summability method, Tauberian theorems\\\rule{0.63cm}{0cm}\emph{\!\!Mathematics Subject Classification:} 03E72, 40G05, 40E05}

\noindent\textbf{Abstract:}
\noindent We introduce Euler summability method for sequences of fuzzy numbers and state a Tauberian theorem concerning Euler summability method, of which proof provides an alternative to that of K. Knopp[Über das Eulersche Summierungsverfahren II, Math. Z. {\bf 18} (1923)] when the sequence is of real numbers. As corollaries, we extend the obtained results to series of fuzzy numbers.

\section{Introduction}\allowdisplaybreaks
\noindent A series $\sum a_n$ is said to be summable by means of the Euler summation method $E_p$ to $s$ if
\begin{eqnarray*}
\lim_{n\to\infty}\frac{1}{(p+1)^n}\sum_{k=0}^n\binom{n}{k}p^{n-k}s_k=s
\end{eqnarray*}
where $p>0$ and $s_n=\sum_{k=0}^na_k$.

Euler summability method in original version for $p=1$ was first introduced by L. Euler to accelerate the convergence of infinite series and then developed for arbitrary values of $p$ by K. Knopp\cite{knopp1,knopp2}. As well as being used in the theory of divergent series to assign sums to divergent series, Euler summability method was applied to various fields of mathematics. In approximation theory, authors used the method to improve the rate of convergence of slowly-converging series and achieved to recover exponential rate of convergence for such series\cite{speed1,speed2,speed3,speed4,speed5}. In function theory, the method was used to generate analytic continuations of functions defined by means of power series and to determine the singular points of functions\cite{analytic1,analytic2,analytic3,analytic4}. For the applications in other fields of mathematics , we refer the reader to \cite{diðer0,diðer1,diðer2,diðer3,diðer4,diðer5}.

Following the introduction of the concept of fuzzy set by Zadeh\cite{zadeh}, fuzzy set theory has developed rapidly and aroused the attention of many mathematicians from different branches. In the branch of analysis in connection with sequences and series, convergence properties of sequences and series have been given and different classes of sequences of fuzzy numbers have been introduced\cite{convergence1,convergence2,convergence3,convergence4,added,convergence5,convergence6}. Besides various summability methods in classical analysis have been extended to fuzzy analysis to deal with divergent sequences of fuzzy numbers and authors have given Tauberian conditions which guarantee the convergence of summable sequences\cite{altýn,canakriesz,canakcesaro,onder,sefa,subrahmanyan,tc,tb3,tr,yavuz3}. In addition to these studies, we now introduce the Euler summability method for fuzzy analysis and prove a Tauberian theorem stating a sufficient condition for an Euler summable sequence of fuzzy numbers to be convergent. This proof also provides an alternative to that given by K. Knopp\cite{knopp2} in case of sequences of real numbers. Additionally, analogues of obtained results are given for series of fuzzy numbers.
\section{ Preliminaries}
A \textit{fuzzy number} is a fuzzy set on the real axis, i.e. u is normal, fuzzy convex, upper semi-continuous and $\operatorname{supp}u =\overline{\{t\in\mathbb{R}:u(t)>0\}}$ is compact \cite{zadeh}.
We denote the space of fuzzy numbers by $E^1$. \textit{$\alpha$-level set} $[u]_\alpha$ of $u\in E^1$
is defined by
\begin{eqnarray*}
[u]_\alpha:=\left\{\begin{array}{ccc}
\{t\in\mathbb{R}:u(t)\geq\alpha\} & , & \qquad if \quad 0<\alpha\leq 1, \\[6 pt]
\overline{\{t\in\mathbb{R}:u(t)>\alpha\}} & , &
if \quad \alpha=0.\end{array} \right.
\end{eqnarray*}
Each $r\in\mathbb{R}$ can be regarded as a fuzzy number $\overline{r}$ defined by
\begin{eqnarray*}
\overline{r}(t):=\left\{\begin{array}{ccc}
1 & , & if \quad t=r, \\
0 & , & if \quad t\neq r.\end{array} \right.
\end{eqnarray*}
Let $u,v\in E^1$ and $k\in\mathbb{R}$. The addition and scalar multiplication are defined by
\begin{eqnarray*}
[u+v]_\alpha=[u]_\alpha +[v]_\alpha=[u^-_{\alpha}+v^-_{\alpha}, u^+_{\alpha}+v^+_{\alpha}], [k u]_\alpha=k[u]_\alpha
\end{eqnarray*}
where $[u]_\alpha=[u^-_{\alpha}, u^+_{\alpha}]$, for all $\alpha\in[0,1]$.

\begin{lem}\label{lemma}\cite{bede} The following statements hold:
\begin{itemize}
\item[(i)] $\overline{0} \in E^{1}$ is neutral element with respect to $+$, i.e., $u+\overline{0}=\overline{0}+u=u$ for all $u \in E^{1}$.
\item [(ii)] With respect to $\overline{0}$, none of $u \neq \overline{r}$, $r\in \mathbb{R}$
has opposite in $E^{1}.$
\item[(iii)] For any $a,b \in\mathbb{R}$ with $a, b \geq 0$ or $a,b \leq 0$ and any $u\in E^{1}$, we have
$(a + b) u = au + bu$. For general $a, b \in \mathbb{R}$, the above property does not hold.
\item[(iv)] For any $a \in \mathbb{R}$ and any $u, v \in E^{1}$, we have
$a (u+ v) = au + av.$
\item[(v)] For any $a, b \in \mathbb{R}$ and any $u \in E^{1}$, we have
$a (b u) = (a b) u.$
\end{itemize}
\end{lem}
The metric $D$ on $E^1$ is defined as
\begin{eqnarray*}
 D(u,v):=
\sup_{\alpha\in[0,1]}\max\{|u^-_{\alpha}-v^-_{\alpha}|,|u^+_{\alpha}-
v^+_{\alpha}|\}.
\end{eqnarray*}
\begin{pr}\cite{bede}
\label{p02} Let $u,v,w,z\in E^1$ and $k\in\mathbb{R}$. Then,
\begin{itemize}
\item [(i)] $(E^1,D)$ is a complete metric space.
\item [(ii)] $D(ku,kv)=|k|D(u,v)$.
\item [(iii)] $D(u+v,w+v)=D(u,w)$.
\item [(iv)] $D(u+v,w+z)\leq D(u,w)+D(v,z)$.
\item [(v)] $|D(u,\overline{0})-D(v,\overline{0})|\leq D(u,v)\leq
D(u,\overline{0})+D(v,\overline{0})$.
\end{itemize}
\end{pr}
A sequence $(u_n)$ of fuzzy numbers  is said to be bounded if there exists $M>0$ such that $D(u_n, \bar{0})< M$ for all $n\in \mathbb{N}$. By  $\ell_{\infty}(F)$, we denote the set of all bounded sequences of fuzzy numbers.

A sequence $(u_n)$ of fuzzy numbers is said to be convergent to $\mu\in
E^1$ if for every $\varepsilon>0$ there exists an
$n_0=n_0(\varepsilon)\in\mathbb{N}$ such that $D(u_n,\mu)<\varepsilon~~\text{for all} ~~n\geq n_0.$

Let $(u_n)$ be a sequence of fuzzy numbers. Then the expression $\sum u_n$ is called a
\textit{series of fuzzy numbers}. Denote $s_n=\sum_{k=0}^nu_k$ for
all $n\in\mathbb{N}$, if the sequence $(s_n)$ converges to a fuzzy
number $\nu$ then we say that the series $\sum u_n$ of fuzzy numbers
converges to $\nu$ and write $\sum u_n=\nu$. We say otherwise the series of fuzzy numbers diverges. Additionally, if the sequence $(s_n)$ is bounded then we say that the series $\sum u_n$ of fuzzy numbers is bounded. By  $bs(F)$, we denote the set of all bounded series of fuzzy numbers.
\begin{rem}\label{change}
Let $(u_n)$ be a sequence of fuzzy numbers. If $(x_n)$ is a sequence of non-negative real numbers, then
\begin{eqnarray*}
\sum\limits_{k=0}^nx_k\sum\limits_{m=0}^ku_m=\sum\limits_{m=0}^nu_m\sum\limits_{k=m}^nx_k
\end{eqnarray*}
 holds by $(iii)$ and $(iv)$ of Lemma \ref{lemma}.
\end{rem}
\section{Main Results}
\begin{defin}
Let $(u_n)$ be a sequence of fuzzy numbers. The Euler means of $(u_n)$ is defined by
\begin{eqnarray}
t^{p}_n=\frac{1}{(p+1)^n}\sum_{k=0}^n\binom{n}{k}p^{n-k}u_k\qquad\qquad (p>0).
\end{eqnarray}
We say that $(u_n)$ is $E_p$ summable to a fuzzy number $\mu$ if
\begin{eqnarray*}
\lim_{n\to\infty}t^{p}_n=\mu.
\end{eqnarray*}
\end{defin}
\begin{thm}\label{theorem1}
If sequence $(u_n)$ of fuzzy numbers converges to $\mu \in E^1$, then $(u_n)$ is $E_p$ summable to $\mu$.
\end{thm}
\begin{proof}
Let $u_n\rightarrow \mu$. For any $\varepsilon>0$ there exists $n_0=n_0(\varepsilon)$ such that $D(u_n,\mu)< \frac{\varepsilon}{2}$ whenever $ n> n_0$, and for $n\leq n_0$ there exists $M > 0$ such that $D(u_n,\mu)\leq  M$. So, by $(iv)$ of Proposition \ref{p02}, we get
\begin{eqnarray*}
D\left( t^{p}_n, \mu\right)&=&D\left(\frac{1}{(p+1)^n}\sum_{k=0}^n\binom{n}{k}p^{n-k}u_k ,\mu\right)
\\&=&D\left(\frac{1}{(p+1)^n}\sum_{k=0}^n\binom{n}{k}p^{n-k}u_k ,\frac{1}{(p+1)^n}\sum_{k=0}^n\binom{n}{k}p^{n-k}\mu\right)
\\&\leq&
\frac{1}{(p+1)^n}\sum_{k=0}^n\binom{n}{k}p^{n-k}D(u_k, \mu)
\\&=&
\frac{1}{(p+1)^n}\sum_{k=0}^{n_0}\binom{n}{k}p^{n-k}D(u_k, \mu) +\frac{1}{(p+1)^n}\sum_{k=n_0+1}^{n}\binom{n}{k}p^{n-k}D(u_k, \mu)
\\&<&
\frac{M}{(p+1)^n}\sum_{k=0}^{n_0}\binom{n}{k}p^{n-k}+ \frac{\varepsilon}{2}\cdot
\end{eqnarray*}
 There also exists $n_1=n_1(\varepsilon)$ such that $\frac{1}{(p+1)^n}\sum\limits_{k=0}^{n_0}\binom{n}{k}p^{n-k}<\frac{\varepsilon}{2M}$ whenever $ n> n_1$. Then we conclude that $D\left( t^{p}_n, \mu\right)<\varepsilon$ whenever $n>\max\{n_0,n_1\}$, which completes the proof.
\end{proof}
An $E_p$ summable sequence is not necessarily convergent. This is clear from the following example.
\begin{ex}
Let $(u_n)$  be a sequence of fuzzy number  such that
$$u_n(t)=
\begin{cases}
t+(-1)^{n+1}, \quad &(-1)^{n}\leq t\leq (-1)^{n}+1\\
-t+(-1)^{n}+2,\quad &(-1)^{n}+1\leq t\leq (-1)^{n}+2\\
0,& (otherwise)
\end{cases}$$
Since $\alpha-$level set of $u_n$ is $[u_n]_{\alpha}=\left[(-1)^n+\alpha, (-1)^{n}+2-\alpha \right]$, we get
{\footnotesize\begin{eqnarray*}
\left[t^p_n\right]_{\alpha}&=&\left[\frac{1}{(p+1)^n}\sum_{k=0}^n\binom{n}{k}p^{n-k}\{(-1)^k+\alpha\}, \frac{1}{(p+1)^n}\sum_{k=0}^n\binom{n}{k}p^{n-k}\{(-1)^k+2-\alpha\}\right]=\left[\frac{(p-1)^n}{(p+1)^n}+\alpha, \frac{(p-1)^n}{(p+1)^n}+2-\alpha\right].
\end{eqnarray*}}
Then sequence of Euler means $(t^p_n)$ converges to fuzzy number $\mu$ defined by
$$\mu(t)=
\begin{cases}
t, \quad &0\leq t\leq 1\\
2-t, \quad &1\leq t\leq 2\\
0,& (otherwise),
\end{cases}$$
since $D\left(t^p_n,\mu\right)=\frac{(p-1)^n}{(p+1)^n}\to 0$ as $n\to\infty$. However sequence $(u_n)$ does not converge to any fuzzy number.
\end{ex}
Now we give a Tauberian theorem stating a condition under which $E_p$ summability of sequences of fuzzy numbers implies convergence. But before stating this theorem, we refer an inequality for the estimation of cumulative distribution function of binomial law.

Let $X_{n,p}$ be a random variable having the binomial distribution with parameters $(n,p)$ and let $P\{X_{n,p}\leq k\}$ represent the CDF of $X_{n,p}$ for $k\in\{0,1,\ldots, n\}$:
\begin{eqnarray*}
P\{X_{n,p}\leq k\}=\sum_{i=0}^k\binom{n}{i}p^i(1-p)^{n-i}.
\end{eqnarray*}
The following inequality concerning approximation of $P\{X_{n,p}\leq k\}$ has been proved recently by Zubkov and Serov\cite{zubkov}.
\begin{thm}\label{CDF}
Let $H(x,p)=x\ln\left(\frac{x}{p}\right)+(1-x)\ln\left(\frac{1-x}{1-p}\right)$, and let increasing sequences $(C_{n,p}(k))$ be defined as follows : $C_{n,p}(0)=(1-p)^n$, $ C_{n,p}(n)=1-p^n$,
\begin{eqnarray*}
C_{n,p}(k)=\Phi\left(sgn(k-np)\sqrt{2nH\!\!\left(\tfrac{k}{n}, p\right)}\ \right)\qquad \text{for}\qquad 1\leq k<n
\end{eqnarray*}
where $\Phi$ is CDF of the standard normal distribution. Then for every $k\in\{0,1,\ldots, n-1\}$ and for every $p\in(0,1)$
\begin{eqnarray*}
C_{n,p}(k)\leq P\{X_{n,p}\leq k\}\leq C_{n,p}(k+1)
\end{eqnarray*}
and inequalities may happen for $k=0$ or $k=n-1$ only.
\end{thm}
\begin{thm}\label{tauberian}
If sequence $(u_n)$ of fuzzy numbers is $E_p$ summable to fuzzy number $\mu$ and  $\sqrt{n} D(u_{n-1}, u_{n})=o(1)$, then $(u_n)$ converges to $\mu$.
\end{thm}
\begin{proof}
Let sequence $(u_n)$ of fuzzy numbers be $E_p$ summable to fuzzy number $\mu$ and $\sqrt{n} D(u_{n-1}, u_{n})=o(1)$ be satisfied. Then $(u_n)$ is also $E_{\ceil{p}}$ summable to fuzzy number $\mu$, by the following fact:
\begin{eqnarray*}
t^{\frac{\ceil{p}-p}{p+1}}_{n}(t_n^p)&=&\frac{1}{\left(\frac{\ceil{p}-p}{p+1}+1\right)^n}\sum_{k=0}^n\binom{n}{k}\left(\frac{\ceil{p}-p}{p+1}\right)^{n-k}\left\{\frac{1}{(p+1)^k}\sum_{m=0}^k\binom{k}{m}p^{k-m}u_m\right\}
\\&=&
\frac{1}{\left(\ceil{p}+1\right)^n}\sum_{k=0}^n\binom{n}{k}\left(\ceil{p}-p\right)^{n-k}\left\{\sum_{m=0}^k\binom{k}{m}p^{k-m}u_m\right\}
\\&=&
\frac{1}{\left(\ceil{p}+1\right)^n}\sum_{m=0}^nu_m\left\{\sum_{k=m}^n\binom{n}{k}\binom{k}{m}\left(\ceil{p}-p\right)^{n-k}p^{k-m}\right\}
\\&=&
\frac{1}{\left(\ceil{p}+1\right)^n}\sum_{m=0}^n\binom{n}{m}u_m\left\{\sum_{k=0}^{n-m}\binom{n-m}{k}\left(\ceil{p}-p\right)^{n-m-k}p^{k}\right\}
\\&=&
\frac{1}{\left(\ceil{p}+1\right)^n}\sum_{m=0}^n\binom{n}{m}{\ceil{p}}^{n-m}u_m=t_n^{\ceil{p}}
\end{eqnarray*}
where $\ceil{\cdot}$ denotes ceiling function. For brevity we take $q=\ceil{p}$ throughout the proof. Now, let prove the convergence of sequence $(u_n)$ to fuzzy number $\mu$. We can write
\begin{eqnarray*}
D(u_n,\mu)\leq D\left(u_n,t^{q}_{(q+1)n}\right)+D\left(t^{q}_{(q+1)n},\mu\right)
\end{eqnarray*}
and since $(u_n)$ is $E_{q}$ summable to fuzzy number $\mu$ we know that $D\left(t^{q}_{(q+1)n},\mu\right)=o(1)$. So it is sufficient to show that $D\left(u_n,t^{q}_{(q+1)n}\right)=o(1)$. By $(iv)$ of Proposition \ref{p02} and by the condition $\sqrt{n} D(u_{n-1}, u_{n})=o(1)$ of the theorem we get
{\small\begin{eqnarray*}
D\left(t^{q}_{(q+1)n}, u_n\right)&\leq& \frac{1}{(q+1)^{(q+1)n}}\sum_{k=0}^{(q+1)n}\binom{(q+1)n}{k}q^{(q+1)n-k}D(u_k,u_n)
\\&=&
o(1)\Bigg\{\frac{1}{(q+1)^{(q+1)n}}\sum_{k=0}^{(q+1)n}\binom{(q+1)n}{k}q^{(q+1)n-k}\left\{\frac{|n-k|}{\sqrt{n}}\right\}\Bigg\}
\\&=&
o(1)\Bigg\{\frac{\sqrt{n}}{(q+1)^{(q+1)n}}\sum_{k=0}^{n}\binom{(q+1)n}{k}q^{(q+1)n-k} -\frac{1}{\sqrt{n}(q+1)^{(q+1)n}}\sum_{k=0}^{n}k\binom{(q+1)n}{k}q^{(q+1)n-k}
\\&&
\qquad\quad+\frac{1}{\sqrt{n}(q+1)^{(q+1)n}}\sum_{k=n+1}^{(q+1)n}k\binom{(q+1)n}{k}q^{(q+1)n-k}-\frac{\sqrt{n}}{(q+1)^{(q+1)n}}\sum_{k=n+1}^{(q+1)n}\binom{(q+1)n}{k}q^{(q+1)n-k}\Bigg\}
\\&=&
{\mathsmaller{\mathsmaller o(1)\sqrt{n}\Bigg\{\sum\limits_{k=0}^{n}\binom{(q+1)n}{k}\left(\frac{1}{q+1}\right)^{k}\left(\frac{q}{q+1}\right)^{(q+1)n-k}-\sum\limits_{k=0}^{n-1}\binom{(q+1)n-1}{k}\left(\frac{1}{q+1}\right)^k\left(\frac{q}{q+1}\right)^{(q+1)n-1-k}}}
\\&&
\qquad\quad\quad{\mathsmaller{\mathsmaller+\sum\limits_{k=n}^{(q+1)n-1}\binom{(q+1)n-1}{k}\left(\frac{1}{q+1}\right)^k\left(\frac{q}{q+1}\right)^{(q+1)n-1-k}-\sum\limits_{k=n+1}^{(q+1)n}\binom{(q+1)n}{k}\left(\frac{1}{q+1}\right)^{k}\left(\frac{q}{q+1}\right)^{(q+1)n-k}}}\Bigg\}
\\&=&
{\mathsmaller {\mathsmaller o(1)2\sqrt{n}\left\{\sum\limits_{k=0}^{n}\binom{(q+1)n}{k}\left(\frac{1}{q+1}\right)^{k}\left(\frac{q}{q+1}\right)^{(q+1)n-k}-\sum\limits_{k=0}^{n-1}\binom{(q+1)n-1}{k}\left(\frac{1}{q+1}\right)^k\left(\frac{q}{q+1}\right)^{(q+1)n-1-k}\right\}}}
\\&=&
o(1)2\sqrt{n}\left\{\sum\nolimits_1-\sum\nolimits_2\right\}.
\end{eqnarray*}}
By Theorem \ref{CDF}, we obtain
\begin{eqnarray*}
\sum\nolimits_1&\leq&C_{{(q+1)n},{\frac{1}{q+1}}}(n+1)
\\&=&
\Phi\left(\sqrt{2(q+1)nH\!\!\left(\frac{n+1}{(q+1)n},\frac{1}{q+1}\right)}\ \right)=\Phi\left(\sqrt{2\ln{\left(\left\{\frac{n+1}{n}\right\}^{n+1}\left\{\frac{qn-1}{qn}\right\}^{qn-1}\right)}}\ \right)
\\&=&
\frac{1}{2}+\frac{1}{2}\text{erf}\left(\sqrt{\ln{\left(\left\{\frac{n+1}{n}\right\}^{n+1}\left\{\frac{qn-1}{qn}\right\}^{qn-1}\right)}}\ \right)
\end{eqnarray*}
and
\begin{eqnarray*}
 \sum\nolimits_2&\geq&C_{{(q+1)n-1},{\frac{1}{q+1}}}(n-1)
\\&=&
{\mathsmaller{\mathsmaller\Phi\left(-\sqrt{2((q+1)n-1)H\!\!\left(\frac{n-1}{(q+1)n-1},\frac{1}{q+1}\right)}\ \right)=\Phi\left(-\sqrt{2\ln{\left(\left\{\frac{(q+1)(n-1)}{(q+1)n-1}\right\}^{n-1}\left\{\frac{(q+1)n}{(q+1)n-1}\right\}^{qn}\right)}}\ \right)}}
\\&=&
\frac{1}{2}-\frac{1}{2}\text{erf}\left(\sqrt{\ln{\left(\left\{\frac{(q+1)(n-1)}{(q+1)n-1}\right\}^{n-1}\left\{\frac{(q+1)n}{(q+1)n-1}\right\}^{qn}\right)}}\ \right).
\end{eqnarray*}
So we get
{\scriptsize\begin{eqnarray*}
D\left(t^{q}_{(q+1)n}, u_n\right)&=& o(1)\sqrt{n}\left\{\text{erf}\left(\sqrt{\ln{\left(\left\{\frac{n+1}{n}\right\}^{n+1}\left\{\frac{qn-1}{qn}\right\}^{qn-1}\right)}}\ \right)+\text{erf}\left(\sqrt{\ln{\left(\left\{\frac{(q+1)(n-1)}{(q+1)n-1}\right\}^{n-1}\left\{\frac{(q+1)n}{(q+1)n-1}\right\}^{qn}\right)}}\ \right)\right\}
\\&=&
 o(1)\frac{2}{\sqrt{\pi}}\sqrt{n}\left\{\sqrt{\ln{\left(\left\{\frac{n+1}{n}\right\}^{n+1}\left\{\frac{qn-1}{qn}\right\}^{qn-1}\right)}}\ +\sqrt{\ln{\left(\left\{\frac{(q+1)(n-1)}{(q+1)n-1}\right\}^{n-1}\left\{\frac{(q+1)n}{(q+1)n-1}\right\}^{qn}\right)}}\ \right\}
\\&=&
o(1)\frac{2}{\sqrt{\pi}}\Bigg\{\underbrace{ \sqrt{\ln\left(\left\{\frac{n+1}{n}\right\}^{n+1}\left\{\frac{qn-1}{qn}\right\}^{qn-1}\right)^n}}_{\longrightarrow \sqrt{\frac{q+1}{2q}}\quad (n\to\infty)}\ +\underbrace{\sqrt{\ln{\left(\left\{\frac{(q+1)(n-1)}{(q+1)n-1}\right\}^{n-1}\left\{\frac{(q+1)n}{(q+1)n-1}\right\}^{qn}\right)^n}}}_{\longrightarrow \sqrt{\frac{q}{2(q+1)}}\quad (n\to\infty)} \Bigg\}=o(1),
\end{eqnarray*}}
which completes the proof.
\end{proof}
\begin{thm}
If sequence $(u_n)$ of fuzzy numbers is $E_p$ summable to fuzzy number $\mu$ and  $\sqrt{n} D(u_{n-1}, u_{n})=O(1)$, then $(u_n)\in\ell_{\infty}(F)$.
\end{thm}
\begin{proof}
Let sequence $(u_n)$ of fuzzy numbers be $E_p$ summable to fuzzy number $\mu$ and $\sqrt{n} D(u_{n-1}, u_{n})=O(1)$ be satisfied. From the proof Theorem \ref{tauberian}, $(u_n)$ is $E_q$ summable to $\mu$ where $q=\ceil{p}$. We also have
\begin{eqnarray*}
D(u_n,\bar{0})\leq D\left(u_n,t^{q}_{(q+1)n}\right)+D\left(t^{q}_{(q+1)n},\mu\right)+D(\mu,\bar{0}).
\end{eqnarray*}
Since $(u_n)$ is $E_q$ summable to $\mu$, $D\left(t^{q}_{(q+1)n},\mu\right)=O(1)$ holds. It is sufficient to show that $D\left(u_n,t^{q}_{(q+1)n}\right)=O(1)$. By hypothesis $\sqrt{n} D(u_{n-1}, u_{n})=O(1)$ of theorem,  we have $D(u_{k}, u_{n})=O(1)\frac{|n-k|}{\sqrt{n}}$. Processing as in the proof of Theorem \ref{tauberian} we obtain
{\footnotesize\begin{eqnarray*}
D\left(u_n,t^{q}_{(q+1)n}\right)=O(1)\frac{2}{\sqrt{\pi}}\Bigg\{\underbrace{ \sqrt{\ln\left(\left\{\frac{n+1}{n}\right\}^{n+1}\left\{\frac{qn-1}{qn}\right\}^{qn-1}\right)^n}}_{\longrightarrow \sqrt{\frac{q+1}{2q}}\quad (n\to\infty)}\ +\underbrace{\sqrt{\ln{\left(\left\{\frac{(q+1)(n-1)}{(q+1)n-1}\right\}^{n-1}\left\{\frac{(q+1)n}{(q+1)n-1}\right\}^{qn}\right)^n}}}_{\longrightarrow \sqrt{\frac{q}{2(q+1)}}\quad (n\to\infty)} \Bigg\},
\end{eqnarray*}}
which yields $D\left(u_n,t^{q}_{(q+1)n}\right)=O(1)$ as $n\to\infty$ and proof is completed.
\end{proof}
To any series of fuzzy numbers there corresponds a sequence of partial sums which is also of fuzzy numbers. So $E_p$ summability of a series of fuzzy numbers can be determined by $E_p$ summability of corresponding sequence of partial sums. As a result we can extend the results above to series of fuzzy numbers. At this point we note that, the converse result extensions from  series of fuzzy numbers to sequences of fuzzy numbers  may not always be possible since there are sequences of fuzzy numbers which can not be represented by series of fuzzy numbers. For details we refer to\cite{yavuz1}.
\begin{defin}
A series $\sum u_n$ of fuzzy numbers is said to be $E_p$ summable to fuzzy number $\nu$ if the sequence of partial sums of the series $\sum u_n$  is $E_p$ summable to $\nu$.
\end{defin}
\begin{cor}
If series $\sum u_n$ of fuzzy numbers converges to fuzzy number $\nu$, then it is $E_p$ summable to $\nu$.
\end{cor}
\begin{cor}\label{alternative}
If series $\sum u_n$ of fuzzy numbers is $E_p$ summable to fuzzy number $\nu$ and  $\sqrt{n}D(u_n, \bar{0})=o(1)$, then $\sum u_n=\nu$.
\end{cor}
\begin{cor}
If series $\sum u_n$ of fuzzy numbers is $E_p$ summable to fuzzy number $\nu$ and  $\sqrt{n}D(u_n, \bar{0})=O(1)$, then $(u_n)\in bs(F)$.
\end{cor}
The proof of Theorem \ref{tauberian} with Corollary \ref{alternative} yields also the following Tauberian theorem of K. Knopp in case of real numbers.
\begin{thm}\cite{knopp2}
If series $\sum a_n$ of real numbers is $E_p$ summable to $s$ and  $a_n=o\left(\frac{1}{\sqrt{n}}\right)$, then  $\sum a_n=s$.
\end{thm}


\begin{thebibliography}{99}

\bibitem{diðer0} B. Altay, F. Ba\c sar, M. Mursaleen, On the Euler sequence spaces which include the spaces $\ell_p$ and $\ell_{\infty}$ I, Information Sciences 176 (2006), 1450--1462.

\bibitem{altýn}Y. Alt{\i}n , M. Mursaleen, H. Alt{\i}nok, Statistical summability $(C; 1)$-for sequences of fuzzy real numbers and a
Tauberian theorem, Journal of Intelligent and Fuzzy Systems 21 (2010), 379--384.

\bibitem{convergence1} H. Alt{\i}nok, R. \c{C}olak, Y. Alt{\i}n, On the class of $\lambda$-statistically convergent difference sequences of fuzzy numbers, Soft Computing 16(6) (2012), 1029--1034.


\bibitem{bede} B. Bede, S. G. Gal, Almost periodic fuzzy-number-valued functions, Fuzzy Sets and Systems 147 (2004), 385--403.

\bibitem{analytic1} J. Boos, Classical and Modern Methods in Summability, Oxford University Press 2000.


\bibitem{speed1} J. P. Boyd, Sum-accelerated pseudospectral methods: the Euler-accelerated sine algorithm, Applied Numerical Mathematics 7 (1991), 287--296.

\bibitem{speed2} J. P. Boyd, A proof, based on the Euler sum acceleration, of the recovery of an exponential (geometric) rate of convergence for the Fourier series of a function with Gibbs phenomenon, in {\it Spectral and High Order Methods for Partial Differential Equations}, Springer 2011.

\bibitem{diðer1} P. Chandra, Multipliers for the absolute Euler summability of Fourier series, Proceedings of the Indian Academy of Sciences-Mathematical Sciences 111(2) (2001), 203--219.

\bibitem{canakriesz} \.{I}. \c{C}anak, On the Riesz mean of sequences of fuzzy real numbers, Journal of Intelligent and Fuzzy Systems 26(6) (2014), 2685--2688.

\bibitem{canakcesaro} \.{I}. \c{C}anak, On Tauberian theorems for Ces\`{a}ro summability of sequences of fuzzy numbers, Journal of Intelligent and Fuzzy Systems 30(5) (2016), 2657--2662.

\bibitem{diðer2} G. D. Dikshit, Absolute Euler Summability of Fourier Series, Journal of Mathematical Analysis and Applications 220 (1998), 268--282.

\bibitem{speed3} J.E. Drummond, Convergence speeding, convergence and summability, Journal of Computational and Applied Mathematics 11(2) (1984), 145--159.

\bibitem{analytic2} R. Estrada, J. Vindas, Exterior Euler summability, Journal of Mathematical Analysis and Applications 388 (2012), 48--60.

\bibitem{analytic3} M. M. Kabardov, On Analytic Continuation of a Hypergeometric Series Using the Euler-Knopp Transformation, Vestnik St. Petersburg University: Mathematics 42(3) (2009), 169--174.

\bibitem{knopp1} K. Knopp, Über das Eulersche Summierungsverfahren, Mathematische Zeitschrift 15 (1922), 226--253.

\bibitem{knopp2} K. Knopp, Über das Eulersche Summierungsverfahren II, Mathematische Zeitschrift 18 (1923), 125--156.

\bibitem{speed4} O. Meronen, I. Tammeraid, Generalized Euler-Knopp method and convergence acceleration, Mathematical Modelling and Analysis 11(1) (2006), 87--94.

\bibitem{diðer3}  M. Mursaleen, F. Ba\c sar, B. Altay, On the Euler sequence spaces which include the spaces $\ell_p$ and $\ell_{\infty}$ II, Nonlinear Analysis 65 (2006), 707--717.

\bibitem{diðer4} M. Mursaleen, Applied summability methods, Springer, 2014.

 \bibitem{speed5} T. Ooura, A generalization of the continuous Euler transformation and its application to numerical quadrature, Journal of Computational and Applied Mathematics 157 (2003), 251--259.

\bibitem{onder}Z. \"{O}nder, S. A. Sezer, \.{I}. \c{C}anak, A Tauberian theorem for the weighted mean method of summability of sequences of fuzzy numbers,  Journal of Intelligent and Fuzzy Systems 28 (2015), 1403--1409.

\bibitem{sefa} S. A. Sezer, \.{I}. \c{C}anak, Power series methods of summability for series of fuzzy numbers and related Tauberian Theorems, Soft Copmuting (2015), Doi: 10.1007/s00500-015-1840-0

\bibitem{analytic4} J. Sondow, Analytic continuation of Riemann's zeta function and values at negative integers via Euler's transformation of series, Proceedings of the American Mathematical Society 120(2) (1994), 421--424.

\bibitem{convergence2} M. Stojakovi\'{c}, Z. Stojakovi\'{c}, Addition and series of fuzzy sets, Fuzzy Sets and Systems 83 (1996), 341--346.

\bibitem{convergence3} M. Stojakovi\'{c}, Z. Stojakovi\'{c}, Series of fuzzy sets, Fuzzy Sets and Systems 160 (2009), 3115--3127.

\bibitem{subrahmanyan} P. V. Subrahmanyam, Ces\`{a}ro summability of fuzzy real numbers, Journal of Analysis 7 (1999), 159--168.

\bibitem{tc} \"O. Talo, C. \c{C}akan, On the Ces\`{a}ro convergence of sequences of fuzzy numbers, Applied Mathematics Letters 25 (2012), 676--681.

\bibitem{tb3}\"O. Talo, F. Ba\c sar, On the Slowly Decreasing Sequences of Fuzzy Numbers,  Abstract and Applied Analysis  (2013), 1--7.

\bibitem{convergence4} \"O. Talo, U. Kadak, F. Ba\c sar, On series of fuzzy numbers, Contemporary Analysis and Applied Mathematics 4(1) (2016), 132--155.
573 207

\bibitem{tr} B. C. Tripathy, A. Baruah, N\"{o}rlund and Riesz mean of sequences of fuzzy real numbers, Applied Mathematics Letters 23 (2010), 651--655.

\bibitem{added} B.C. Tripathy, P.C. Das, On convergence of series of fuzzy real numbers, Kuwait Journal of Science and Engineering , 39(1A) (2012), 57--70.

\bibitem{convergence5} B. C. Tripathy, M. Sen, On fuzzy I-convergent difference sequence space, Journal of Intelligent and Fuzzy Systems 25(3) (2013), 643--647.

\bibitem{convergence6} B. C. Tripathy, N. L. Braha, A. J. Dutta, A new class of fuzzy sequences related to the $\ell_p$ space defined by Orlicz function, Journal of Intelligent and Fuzzy Systems, 26(3) (2014), 1273--1278.

\bibitem{diðer5} H. Walk, Almost sure Ces\`{a}ro and Euler summability of sequences of dependent random variables, Archiv der Mathematik 89 (2007), 466--480.

\bibitem{yavuz1} E. Yavuz, H. \c{C}o\c{s}kun, Tauberian theorems for Abel summability of sequences of fuzzy numbers, AIP Conference Proceedings 1676 (2015), doi: 10.1063/1.4930505

\bibitem{yavuz3} E. Yavuz, H. \c{C}o\c{s}kun, On the logarithmic summability method for sequences of fuzzy numbers, Soft Computing (2016), doi: 10.1007/s00500-016-2156-4

\bibitem{zadeh} L. A. Zadeh, Fuzzy sets, Information and Control 8 (1965), 29--44.

\bibitem{zubkov} A. M. Zubkov, A. A. Serov, A complete proof of universal inequalities for the distribution function of the binomial law, Theory of Probability \& Its Applications 57(3) (2013), 539-544.
\end{thebibliography}
\end{document}